\pdfoutput=1
\documentclass{amsart}
\usepackage{amsmath,amssymb}
\usepackage{hyperref}
\usepackage{amsrefs}
\usepackage{graphicx,color}
\usepackage{pict2e}
\usepackage{enumitem}

\newlist{steps}{enumerate}{1}
\setlist[steps]{label=Step~\arabic*, ref=Step~\arabic*}
\newtheorem{theorem}{Theorem}
\newtheorem{proposition}{Proposition}
\newtheorem{lemma}{Lemma}
\newtheorem{fact}{Fact}
\newtheorem{corollary}{Corollary}

\theoremstyle{definition}
\newtheorem{definition}{Definition}

\newtheorem{notation}{Notation}
\newtheorem{framework}{Framework}

\theoremstyle{remark}
\newtheorem{remark}{Remark}

\newcommand{\ri}{\operatorname{RI}}
\newcommand{\Sj}{S^-_{\textup{join}}}
\newcommand{\Ts}{T_{\textup{split}}}
\newcommand{\seq}{\operatorname{seq}}

\begin{document}
\title[Genera of two-component alternating links]{Genera of two-component alternating links}
\author{Noboru Ito}
\author{Nodoka Kawajiri}
\address{Department of Mathematics, Shinshu University,  Wakasato 4-17-1, Nagano, 380-8553, Japan.  
}
\email{nito@shinshu-u.ac.jp}
\date{February 5, 2026}
\maketitle
\begin{abstract}
We extend the equality-type results of Ito--Takimura and Kindred
for the non-orientable genera of alternating knots
to the setting of two-component alternating links.
We show that, for such links, 
a unified quantity capturing both orientable and non-orientable genera   
is completely determined by the splice sequence realizing the splice-unknotting number up to an explicit correction term.
\end{abstract}
\section{Introduction}\label{sec:Intro}
The determination of orientable and non-orientable genera of knots and links
has long been a classical and central problem in topology.
For knots, the orientable genus admits a well-established theory,
beginning with Seifert’s construction \cite{Seifert1935, Kauffman1983} and extending to equality-type results
for alternating knots via the Alexander polynomial \cite{Murasugi1958, Crowell1959},
as well as powerful general tools such as Heegaard Floer homology.
In contrast, the non-orientable genus, also known as the crosscap number,
has historically resisted a comparable description in general settings.
Aside from specific classes—for example, those studied by
Hatcher--Thurston, Teragaito  \cite{Teragaito2004}, Hirasawa--Teragaito \cite{HirasawaTeragaito2006},
and Ichihara--Mizushima \cite{IchiharaMizushima2010}—no equality-type framework had been available
for several decades.
A notable breakthrough was achieved independently by
Ito--Takimura \cite{ItoTakimura2020b, ItoTakimura2020v} and Kindred \cite{Kindred2020},
who showed that for any prime alternating knot,
the non-orientable genus is exactly determined by the splice-unknotting number.

A fundamental difficulty remains, however, in extending such results
from knots to links.
In the splittable link case, the problem can be reduced to the knot components,
and equality-type bounds follow by combining classical knot inequalities
with elementary Euler characteristic considerations.
In particular, for split two-component links,
Zhang \cite{Zhang2008} obtained sharp formulas under a convention based on
connected non-orientable spanning surfaces.
For non-splittable links, however, such a reduction, as in the split case, is no longer available.
Even for two-component links, the interaction between components
produces genuinely new phenomena,
and no equality-type description analogous to the knot case
has previously been known.
Hence, the purpose of this paper is to address this remaining case 
by studying two-component alternating links,
and to show that their orientable and non-orientable genera
are governed by the splice-unknotting number
up to an explicit and computable correction term.

The relationship between the splice-unknotting number $u^-(K)$,
introduced in \cite{Itotakimura2018}, and the non-orientable genus (crosscap  number) $C(K)$ is well understood (Fact~\ref{thm:ItoTakimuraUC}).      
\begin{fact}[{Ito-Takimura \cite{ItoTakimura2020v}}, {Kindred \cite{Kindred2020}}]\label{thm:ItoTakimuraUC}
For any prime alternating knot $K$, $u^-(K)$ $=$ $C(K)$.  
\end{fact}

We extend Fact~\ref{thm:ItoTakimuraUC} to the two-component case (Theorem~\ref{thm:uL}) by introducing definitions corresponding to the link case  (Definitions~\ref{def:uL} and ~\ref{def:mdivisor}).   
\begin{definition}\label{def:uL}
Let $L$ be a non-splittable two-component link and $D$ an arbitrary diagram of $L$.
Let $\seq(D)$ be a splice sequence consisting only of splices of types $\Sj$, $S^-$, and $\ri^-$ (Definition~\ref{def:splice} and Figure~\ref{fig:splice2}) to obtain a circle with no crossings from $D$.
Let $\#S^-(\seq(D))$ denote the number of splices of type $S^-$ appearing in $\seq(D)$.
Then we define
\[
u^-(D):=\min_{\seq(D)} \#S^-(\seq(D)),\qquad
u^-(L):=\min_D u^-(D),
\]
where the minimum is taken over all splice sequences $\seq(D)$ as above and all diagrams $D$ of $L$. 
\end{definition}
Here the symbol $\ri^-$ is regarded as a splice in the sense of
Figure~\ref{fig:splice2}, rather than as a Reidemeister move.  
\begin{remark}\label{rmk:uL-uniqueSj}
In $\seq(D)$, once a splice of type $\Sj$ is applied, the resulting diagram becomes one-component, and hence no further splice of type $\Sj$ is applicable; in particular, any splice sequence $\seq(D)$ realizing $u^-(D)$ contains exactly one occurrence of $\Sj$ (Lemma~\ref{lem:UniqueSjoin}).
\end{remark}  
\begin{definition}\label{def:mdivisor}    
Any alternating knot $K$ admits a prime decomposition
$K=K_1$ $\#\cdots\#$ $K_n$ in which each $K_i$ is alternating~\cite{Menasco1984}.     
Letting $B(K_i)$ be as in Definition~\ref{def:BL}, $B(K_i)$ $=$ $u^- (K_i)$ or $u^-(K_i)-1$ \cite{ItoTakimura2020b}.  If the latter equation holds, we call $K_i$ \emph{a negative factor}.  Suppose that the number of negative factors is $m$.  Then let $[m]=m$ if $m=0, 1$ and $[m]=m-1$ otherwise.   
\end{definition}
\begin{definition}\label{def:genus}
Given an $n$-component link $L$, we define
\[
C(L)=\min \left\{
\begin{aligned}
& 2-\chi(\Sigma)-n \\
& \mid \Sigma~\text{is a connected non-orientable spanning surface with } \partial\Sigma=L
\end{aligned}
\right\},
\]
where $\chi(\cdot)$ denotes the Euler characteristic.  
We define the \emph{genus} of $L$ to be
\[
\min\{\,C(L),\,2g(L)\,\},
\]
where $g(L)$ denotes the orientable genus 
and $C(L)$ denotes the non-orientable genus of $L$ in $\mathbb{R}^3$.
\end{definition}
\begin{remark}
The quantity $\min\{C(L),2g(L)\}$ is introduced to treat
orientable and non-orientable spanning surfaces in a unified manner.  
\end{remark}
In the following, throughout this paper, $L$ denotes a non-splittable two-component link.\footnote{For splittable two-component links, we have results in \cite{Zhang2008}.}  
    
Before stating a main result (Theorem~\ref{thm:uL}), we restrict to alternating diagrams from this point on.  We emphasize that  
this alternating assumption is used only to guarantee that the maximal Euler characteristic of spanning surfaces is realized by a state surface and no other part of the argument relies on alternation.   
Our main result is as follows.  
\begin{theorem}\label{thm:uL}
Let $L$ be a non-splittable two-component alternating link and let $c$ be a crossing between the two components of an alternating link diagram $D$ of $L$.  Let $K$ and $K'$ be knots obtained from splices of two different types on $c$.  
Suppose that the maximal Euler characteristic of spanning surfaces of $K$ is greater than that of $K'$ and $K$ includes $m$ negative factors.  Let $u^-(L)$, $[m]$, $C(L)$, and $g(L)$ be as in Definitions~\ref{def:uL},  \ref{def:mdivisor}, and \ref{def:genus}.  
Then 
\[
\min \{ C(L), 2g(L) \} = u^-(L) - [m].  
\] 
\end{theorem}

Before proceeding further, we briefly review some literature on knot/link genera, both orientable and especially non-orientable.  
The problem to determine orientable/non-orientable knot genus has a long history.  In the orientable setting, Seifert \cite{Seifert1935} provided a method for constructing orientable surfaces bounded by a knot.   In the four-dimensional case, Milnor famously conjectured the genus of knots arising from complex algebraic curves, known as the Milnor conjecture.  Over time, a general theory has developed in the orientable knot genus, supported by systematic techniques and theoretical progress.  

In contrast, for non-orientable surfaces measured by the \emph{crosscap number}, also known as the non-orientable genus,  no general approach via equality-type formulas had been available for a long time.     
Since the study of surfaces in complements of two-bridge knots by Hatcher and Thurston \cite{HatcherThurston1985}, computations of crosscap numbers have been carried out for specific classes of knots: the torus knots (Teragaito \cite{Teragaito2004}), the two-bridge knots (Hirasawa-Teragaito \cite{HirasawaTeragaito2006}), and many pretzel knots (Ichihara-Mizushima \cite{IchiharaMizushima2010}).  

On the other hand, general bounds for arbitrary knots have been developed.  For example, Murakami-Yasuhara \cite{MurakamiYasuhara1995} gives strong upper bounds in terms of the minimal crossing number $n(K)$, and Kalfagianni-Lee \cite{KalfagianniLee2016} uses the twist number $t(D)$ of a knot diagram $D$ to obtain two-sided bounds.  That is, the crosscap number also exhibits a strong connection to quantum knot polynomial such as the Jones polynomial.  For example, if the knot is alternating, $n(K)$ coincides with the span of the Jones polynomial, and $t(D)$ is equal to the sum $|a_{M-1}|+|a_{m+1}|$, where $a_{M-1}$ and $a_{m+1}$ are the coefficients of the second-highest and second-lowest degrees of the Jones polynomial, respectively.  A relatively general equality-type result has become available in a recent year through the independent works of Ito-Takimura \cite{ItoTakimura2020v} and Kindred \cite{Kindred2020}, which gave exact formulas for the crosscap number of the alternating knots.  
In this paper, we extend this equality-type approach to the setting of two-component alternating links.  

In what follows, we review a statement from \cite{ItoTakimura2020v, Kindred2020} (Fact~\ref{thm:ItoTakimuraUC}), then present the main result (Theorem~\ref{thm:uL}) of this paper, which extends their formula to the case of two-component alternating links.  

This paper not only extends the splice-unknotting number to two-component links, but also obtains a Clark-type inequality for arbitrary links.    
For split links, one can obtain a Clark-type bound by applying original Clark inequality 
to each knot component and accounting for the change of the Euler characteristic
when passing to a connected spanning surface.  Further, for split two-component links, the result of Zhang~\cite{Zhang2008} obtains 
sharper estimates\footnote{Her paper is written under a different convention for the crosscap number, see Remark~\ref{rmk:NotationCap}.}.  For the non-splittable-link case, however, such a reduction to knot components
is no longer available.
Under the convention as in Definition~\ref{def:genus}, which measures genus via the Euler characteristic
of connected spanning surfaces,
we record a Clark-type inequality for links in this setting (Proposition~\ref{prop:ExtendClark}).

We conclude this section with the outline of this paper.   In Section~\ref{sec:DefNot}, we present the necessary definitions and notations.  
In Section~\ref{sec:statements}, we state a more detailed version of Theorem~\ref{thm:ItoTakimuraUC} of Section~\ref{sec:Intro} along with Proposition~\ref{Subprop}, summarizing the main results of this paper.  In Section~\ref{sec:proofProp}, we  prove  Proposition~\ref{Subprop}, and in Section~\ref{sec:proof} we prove  Theorem~\ref{Submainthm}.  Section~\ref{sec:table} presents a table of computed values, and  Section~\ref{sec:example} illustrates our results with figures.    
\section{Definitions and notations}\label{sec:DefNot}
\begin{definition}
Given a link diagram, an operation as shown in Figure~\ref{fig:splice1} is called a \emph{splice} of a crossing.  Note that there are exactly two splices on each crossing of a link diagram.  
\begin{figure}[htpb]
    \includegraphics[width=10cm]{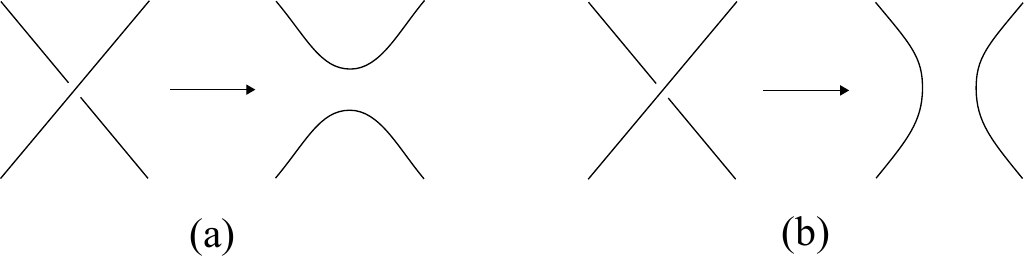}
    \caption{Splices}
    \label{fig:splice1}
\end{figure}
\end{definition}

\begin{definition}\label{def:splice}
If a deformation at a crossing $c$ as in Figure~\ref{fig:splice2} ~(a) is not Figure~\ref{fig:splice2} ~(d), it is called a splice of type $S^-$.  If the crossing $c$ consists of two paths, each of which belongs to the same component, the number of components is unchanged and the splice is called a splice of type $S^-$, or simply  an $S^-$.    
If the crossing $c$ consists of different components, the number of components is decreased and the splice is called a splice of type $\Sj$; we also denote it by $\Sj$ simply.  
After applying an $S^-$-splice we may reorient the resulting diagram arbitrarily, since our arguments do not depend on the choice of orientation.  
Let \(T_\text{split}\) be a splice as in Figure~\ref{fig:splice2}~(b), which preserves the orientation outside  the neighborhood where the splice is applied.  If we apply a single \(T_\text{split}\) to a one-component diagram, the number of components increases.    
Figure~\ref{fig:splice2}(c) shows the special case of $\Ts$ that splits off a $1$-gon.
In the counting arguments below, each $\ri^-$-reduction is interpreted as producing exactly one such $\Ts$ in the associated state.  The symbol $\ri^-$ denotes a  move as in Figure~\ref{fig:splice2}~(d).  
None of the above splices depends  on the over/under information of any crossing.    
\begin{figure}[htpb]
    \includegraphics[width=\linewidth]{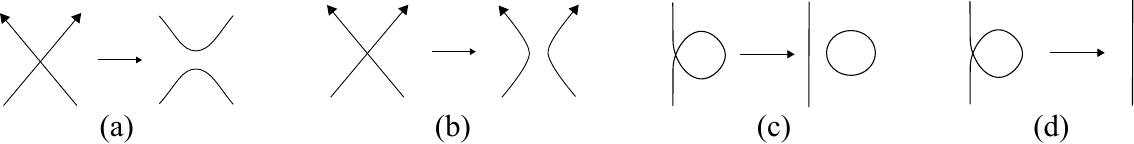}
    \caption{(a) type $S^-$ or $\Sj$, (b) type $\Ts$, (c) type $\Ts$ corresponding to $\ri^-$, (d) type $\ri^-$}
    \label{fig:splice2}
    \end{figure}
\end{definition}

\begin{definition}
A regular projection of a link $L$ into $\mathbb{R}^2$ is called a \emph{link diagram} $D_L$.   For any link diagram $D_L$, 
$S_{D_L}$ denotes the configuration of circles in the plane obtained by choosing a splice at each  crossing of $D_L$.  
This configuration is called a \emph{state} of $D_L$, and each circle in a state is referred to as a \emph{state circle}.   We write $|S_{D_L}|$ for the number of state circles in $S_{D_L}$.   In particular, the state $S_u$ refers to the one obtained from $D_L$ by applying splices that realize its $u^-$ or $u^-_2$.  The surface $\Sigma_u$ is then constructed from $S_u$ by attaching half-twisted bands that recover the original crossings of $D_L$ \footnote{Note that $\Sigma_u$ is not unique, but its Euler characteristic $\chi(\Sigma_u)$ is.} and is called a \emph{state surface}\footnote{The state surfaces have independently been considered previously in Pabiniak, Przytycki, and Sazdanovi\'{c} \cite{PabiniakPrzytyckiSazdanovic2009} and in Ozawa \cite{Ozawa2011}.}.  
Which of $u^-(D_L)$ or $u^-_2(D_L)$ is realized by $\Sigma_u$ will be clear from the context.        
    \begin{figure}[htpb]
\includegraphics[width=12cm]{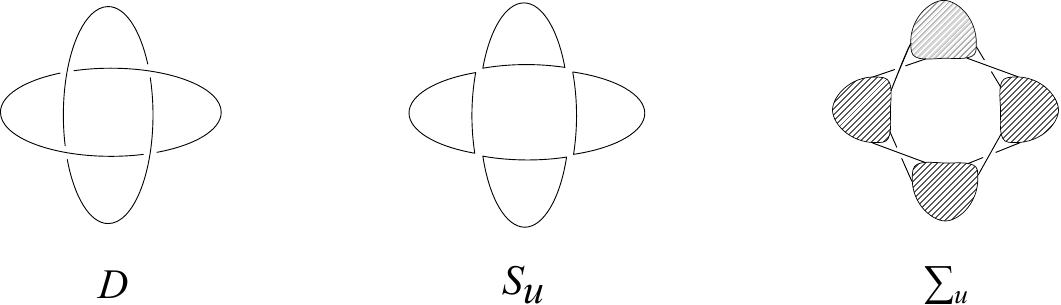}
    \caption{\(D, S_u, \Sigma_u\)}
\end{figure}
If we shall specify $u^-(D_L)$ or $u^-_2 (D_L)$, we may write $\Sigma_{u^-(D_L)}$ or $\Sigma_{u^-_2 (D_L)}$ to present $\Sigma_u$ of $u^-(D_L)$ or $\Sigma_{u^-_2 (D_L)}$.    
\end{definition}    
\begin{definition}\label{def:u}
Let $L$ be a non-splittable two-component link and $D$ an arbitrary diagram of $L$.
Let $\seq(D)$ be a splice sequence consisting only of splices of types $\Sj$, $S^-$, and $\ri^-$
(Definition~\ref{def:splice} and Figure~\ref{fig:splice2})
 to obtain a circle with no crossings from $D$.  
Let $\#S(\seq(D))$ denote the total number of splices of types $S^-$ and $\Sj$ appearing in $\seq(D)$.  
Then we define
\[
u^-_2(D):=\min_{\seq(D)} \#S(\seq(D)),\qquad
u^-_2(L):=\min_D u^-_2(D),
\]
where the minimum is taken over all splice sequences $\seq(D)$ as above and all diagrams $D$ of $L$.
\end{definition}
\begin{remark}\label{rmk:uu2}
For any diagram $D$ of a non-splittable two-component link, we have
\[
u^-_2(D)=u^-(D)+1,
\]
since any splice sequence realizing $u^-(D)$ or $u^-_2(D)$ contains exactly one occurrence of a splice of type $\Sj$ (Lemma~\ref{lem:UniqueSjoin}), and $u^-_2$ counts this splice in addition to the splices of type $S^-$.    
\end{remark}
\begin{definition}[{Adams-Kindred minimal genus algorithm \cite{AdamsKindred2013}}]\label{def:AKalg}
We define an algorithm to obtain a surface whose boundary is a knot.  
    \begin{enumerate}
    \item Find the smallest $m$ for which a given link diagram contains a $m$-gon in the diagram, and choose one of $m$-gons.       
    \item Splice $m$ crossings in the $m$-gon:
        \begin{enumerate}
            \item Case~$m=1$: splice the crossing on this $1$-gon so that this region becomes a circle as in Figure~\ref{fig:AK}~(i).  
            \item Case~$m=2$: splice the crossing on this $2$-gon so that this region becomes a circle as in Figure~\ref{fig:AK}~(ii).  
            \item Case~$m=3$: create two branches.  For one branch, splice the crossing on this $3$-gon so that this region becomes a circle; for the other branch, splice the three crossings the opposite way (Figure~\ref{fig:AK}(iii)).  
        \end{enumerate}
    \item Repeat until each branch reaches a link diagram without crossings.  
    \item Of all state surfaces obtained from these states, choose the one with the maximal Euler characteristic.  
    \end{enumerate}
    \begin{figure}[htpb]
        \centering
        \includegraphics[width=9cm]{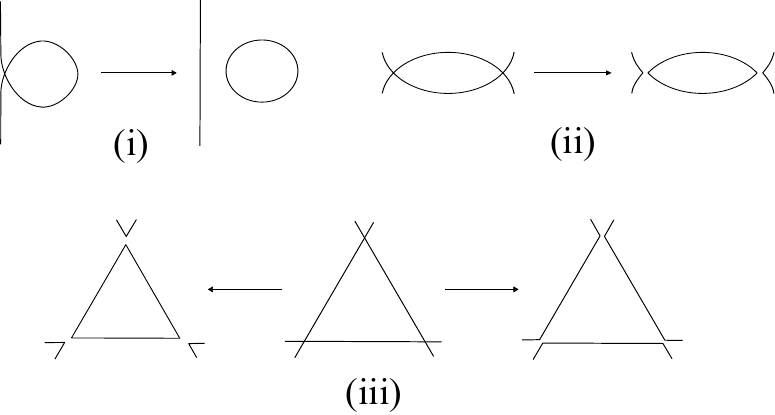}
        \caption{Splices used in Adams-Kindred  algorithm}
        \label{fig:AK}
    \end{figure}  
For a link diagram, let $\Sigma_{AK}$ be a spanning surface obtained from Adams-Kindred algorithm.  Note that there exists a spanning surface $\Sigma_{AK}$ with the maximal Euler characteristic if the diagram is alternating.  
\end{definition}

\begin{definition}
\label{def:BL}
Given a link diagram $D$, we apply  deformations $B$ and splices $\ri^-$ to obtain the circle with no crossings, and the resulting ordered operations gives a sequence.  Let $\seq_B (D)$ denote it.    
Then $\# B (\seq_B (D))$ denotes the number of deformations $B$ in $\seq_B (D)$.  
Let 
    \[B (D) := \min_{\seq_B (D)} \# B (\seq_B (D)), \quad B(L):= \min_{D} B (D). \] 
    \begin{figure}[htpb]
        \centering
        \includegraphics[width=6cm]{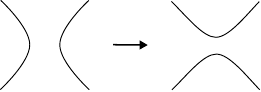}
        \caption{Deformation $B$}
        \label{fig:B}
    \end{figure}
\end{definition}
\begin{remark}
The difference between the two numbers  $B(K)$ and $u^-(K)$ detects negative factors (Fact~\ref{fact:ITB}).  
\end{remark}
\begin{fact}[{Ito-Takimura \cite{ItoTakimura2020b}}]\label{fact:ITB}
For any alternating knot $K$, 
\begin{enumerate}
\item $B(K)=C(K)-1$ if and only if $C(K)=2g(K)+1$, 
\item $B(K)=C(K)$ if and only if $C(K) \neq 2g(K)+1$.  
\end{enumerate}
\end{fact}
Fact~\ref{thm:ItoTakimuraUC} and Fact~\ref{fact:ITB} imply Corollary~\ref{cor:uB}.  
\begin{corollary}\label{cor:uB}
For any prime alternating knot $K$, 
\begin{enumerate}
\item $B(K)=u^-(K)-1$ if and only if $C(K)=2g(K)+1$, 
\item $B(K)=u^-(K)$ if and only if $C(K) \neq 2g(K)+1$.  
\end{enumerate}
\end{corollary}
Note also Fact~\ref{fact:cSum}.   
\begin{fact}[{Murakami-Yasuhara \cite{MurakamiYasuhara1995}}]\label{fact:cSum}
For any connected sum of two knots $K$ and $K'$, 
$C(K \# K')=C(K)+C(K')$ if and only if $C(K)=\min\{C(K), 2g(K)\}$ and $C(K')=\min\{C(K'), 2g(K')\}$.  
\end{fact}
Therefore, the difference $B(K)$ and $u^-(K)$ counts the number of negative factors.  
\begin{definition}[Inter-component $2$-gon $\Sj$]
Let $D$ be a two-component alternating link diagram that contains a $2$-gon consisting of two edges belonging to distinct components.  
If a $2$-gon consists of two edges belonging to distinct components, we call it a \emph{inter-component $2$-gon}.  
For $D$, we assign an orientation to this $2$-gon as in Figure~\ref{fig:muki}, in such a way that the disk bounded by the $2$-gon does not admit an orientation compatible with the edge orientations.  By definition, a  splice at any of both crossings of the $2$-gon, in a manner inconsistent with the given  orientation, yields $\Sj$.  
Then we refer to the operation of applying $\Sj$ to a crossing formed by two edges of the inter-component $2$-gon as an \emph{application \(\Sj\) to an inter-component $2$-gon}.  
    \begin{figure}[htpb]
    \centering
    \includegraphics[width=6cm]{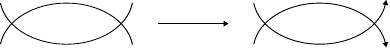}
    \caption{Orientations of an inter-component $2$-gon.  An inconsistent splice yields $\Sj$.}
    \label{fig:muki}
\end{figure}
\end{definition}
\begin{remark}[Convention on $u^-$ and $u^-_2$]\label{rmk:uu2Convention}
By Lemma~\ref{lem:UniqueSjoin}, every splice sequence $\seq(D)$ realizing the minimum in the definition of $u^-(D)$ or $u^-_2(D)$ contains exactly one occurrence of a splice of type $\Sj$.
Here ``exactly one'' refers to the occurrence within the sequence $\seq(D)$, not to a distinguished crossing of $D$.  
Hence, the two invariants $u^-$ and $u^-_2$ differ only by the counting convention:
$u^-_2$ counts splices of types $S^-$ and $\Sj$, while $u^-$ counts only splices of type $S^-$.  
\end{remark}
\section{Main result}\label{sec:statements}
Theorem~\ref{thm:uL} is implied by Theorem~\ref{Submainthm} and the definitions of $u^-(L)$ (Definition~\ref{def:uL}) and $u^-_2 (L)$ (Definition~\ref{def:u})\footnote{We nevertheless arrange formulas for both $u^-$ and $u^-_2$, since a future generalization may require either of them; see  Remark~\ref{rmk:uu2}.   
}.  
\begin{theorem}\label{Submainthm}
Given an alternating two-component link \(L\), let \(D_L\) be an alternating link diagram of $L$.  Then we have  
\begin{align*}
\min \{ C(L), 2g(L) \}&=u^-(D_L)-[m]\\
&=u^-_2(D_L)-1-[m].
\end{align*}
In particular, if $C(L) \le 2g(L)$, 
\begin{align*}
C(L)=u^-(D_L)-[m], \\
C(L)=u^-_2(D_L)-1-[m].
\end{align*}
If $C(L)>2g(L)$, 
\begin{align*}
g(L)=\frac{u^-(D_L)-[m]}{2},\\
g(L)=\frac{u^-_2(D_L)-[m]-1}{2}.
\end{align*}
\end{theorem}
The proof of Theorem~\ref{Submainthm} requires a technical statement
under an additional  assumption, formulated as
Proposition~\ref{Subprop} below.
\begin{proposition}
\label{Subprop}
For any alternating link diagram $D_L$ of a two-component alternating link $L$, suppose that the condition $(\ast)$ holds.  
\begin{center}
$(\ast)$ : $D_L$ has a $2$-gon whose two edges belong to distinct components.  
\end{center}
Under this additional assumption $(\ast)$, the following holds.  
\begin{align*}
\min \{ C(L), 2g(L) \}
&=u^-(D_L)-[m]\\
&=u^-_2(D_L)-1-[m].
\end{align*}

In particular, if $C(L) \le 2g(L)$, 
\begin{align*}
C(L)=u^-(D_L)-[m],\\
C(L)=u^-_2(D_L)-1-[m].
\end{align*}
If $C(L)>2g(L)$, 
\begin{align*}
g(L)=\frac{u^-(D_L)-[m]}{2},\\
g(L)=\frac{u^-_2(D_L)-[m]-1}{2}.
\end{align*}
\end{proposition}
\begin{remark}
Proposition~\ref{Subprop} is a technical form of
Theorem~\ref{Submainthm} with  a specific  assumption.
The proof of Theorem~\ref{Submainthm} consists in reducing the general case
to the setting of Proposition~\ref{Subprop}.
\end{remark}
\section{A proof of Proposition~\ref{Subprop}}\label{sec:proofProp} 
Before proceeding to the proof, we collect the counting identities
for treating $u^-$ and $u^-_2$ uniformly.    Note that the counting in Lemma~\ref{lem:UnifiedCounting} arises going along the plan of Framework~\ref{framework}.   
\begin{lemma}[Counting for $u^-$ and $u^-_2$]
\label{lem:UnifiedCounting}
Let $D_L$ be an alternating diagram with $n$ crossings of a two-component link $L$.  
Fix a splice sequence $\seq(D_L)$ realizing $u^-_2(D_L)$ (resp.~$u^-(D_L)$).  
After the unique splice of type $\Sj$, the diagram becomes one-component, and the remaining part of $\seq(D_L)$ may be viewed as a splice sequence on this resulting knot diagram; hence the same sequence can be used to compare $u^-(D_L)$ and $u^-_2(D_L)$ only by changing the counting convention (Remark~\ref{rmk:uu2Convention}).

Let $D_K$ be the knot diagram obtained after applying the unique occurrence of splice \footnote{The uniqueness of this splice of type $\Sj$ is obtained from Lemma~\ref{lem:UniqueSjoin}.  
} of type $\Sj$ in $\seq(D_L)$ and then eliminating all $1$-gons by $\ri^-$'s.       
Then $\seq(D_L)$ also realizes  $u^-(D_L)$ (resp.~$u^-_2(D_L)$) and the following identities hold.
\begin{enumerate}
\item Exactly one splice of type $\Sj$ appears in each sequence realizing
$u^-(D_L)$ or $u^-_2(D_L)$.
\item The $u^-(D_L)$ and  $u^-_2(D_L)$ satisfy
\[
u^-_2(D_L)=u^-(D_L)+1=u^-(D_K)+1.
\]
\item 
$\# \Ts^{\ri^-}$ denotes the number of splices of type $\Ts$
that correspond to $\ri^-$ (Figure~\ref{fig:splice2}~(c))
and occur in the construction of a state, and $\# {\Ts^{\ri^-}}'$ denotes the number of $\ri^-$-reductions from $D_L$ to $D_K$.
Here the superscript ``$\ri^-$'' is a label indicating correspondence, not an exponent.  

By construction,   
\[
n=\# {\Ts^{\ri^-}}' +u^-(D_K)+1.
\]
\item $|S_u|$ denotes the number of state circles in a state $S_u$ arising from a splice sequence realizing $u^-(D_L)$, and satisfies
\[
|S_u|=\# \Ts^{\ri^-} +1.
\]
\end{enumerate}
\end{lemma}
\subsection{Common framework for proofs}\label{sec:Common}
\begin{framework}\label{framework}
By Lemmas~\ref{lem:UniqueSjoin} and \ref{lem:Inter2gon}, we proceed as follows: 
\begin{steps}
\item \label{step:J}
Choose an inter-component $2$-gon and apply $\Sj$ at one of the two crossings of it.   
By Lemma~\ref{lem:Inter2gon}, we may choose a realizing sequence for $u^-(D_L)$ (or $u^-_2 (D_L)$) whose splice of type $\Sj$ is exactly one; by Lemma~\ref{lem:UniqueSjoin}, it is the only one splice of type $\Sj$.   
Apply the splice of type \(\Sj\) to an inter-component $2$-gon.  
\item \label{step:ri} Apply successive splices of type $\ri^-$ to eliminate all $1$-gons, obtaining a knot projection of an alternating knot $K$.  
\item \label{step:U} Use a splice sequence of $S^-$'s and $\ri^-$'s realizing $u^-(K)$ (Fact~\ref{RealizingUK}).  
\end{steps}
\end{framework}
\noindent
In the following Sections~\ref{sec:case-u}--\ref{u2Proof}, we apply Framework~\ref{framework} to reduce the computation of the Euler characteristic \(\chi(L)\) to the counting identities in Lemma~\ref{lem:UnifiedCounting}; the distinction between prime and non-prime cases concerns the resulting knot diagram \(D_K\).  
\begin{lemma}[Uniqueness of $\Sj$ in a realizing splice sequence]
\label{lem:UniqueSjoin}
For a fixed two-component link diagram $D$, any splice sequence realizing \(u^- (D) \) or \(u^-_2 (D)  \) contains exactly one occurrence of a splice of type $\Sj$.    
\end{lemma}
\begin{proof}
Given a two-component link diagram $D$, in order to realize  $u^-(D)$ or $u^-_2 (D)$, $D$ is transformed into a circle $O$ with no crossings by splices of type $S^-$, type $\Sj$, and type $\ri^-$.  
In this proof, we shall use the unified symbol to refer ``type $S$'' letting $S=S^-, \Sj$.   
Let $\seq(D)$ denote such a fixed splice sequence.  

Note that any $\ri^-$ affects only a local disk including a $1$-gon of the diagram.  Thus outside the affected disk, the curve remains unchanged, including its orientation.     
Hence, whenever the application of $\ri^-$ does not obstruct subsequent splices, it may be postponed.  
In particular, it is possible to defer the splices of type $\ri^-$ until after the splices of type $S$ have been applied.  Reordering the sequence $\seq(D)$ by deferring the application of splices type $\ri^-$  yields a new sequence $\seq'(D)$, in which the number of splices of type $S$ remains the same as in $\seq(D)$.  
Thus, it suffices to consider the case where $\ri^-$ splices appear at the end of the sequence, so that the sequence takes the form 
\[
S \cdots S \ri^- \cdots \ri^-.  
\]
Among the splices involved, $\Sj$ is the only one that changes the number of components, whereas $S^-$ and $\ri^-$ preserve the number of components.  Since the goal is to transform a two-component diagram $D$ into a circle $O$, the sequence must include at least one $\Sj$.  

Moreover, once a single splice of type $\Sj$ has been applied, the diagram becomes one component.  As the $\Sj$ splice is defined only for diagrams with at least two components, no further $\Sj$ splices are applicable.  Hence exactly one $\Sj$ occurs in the sequence.  
\end{proof}
\begin{lemma}[Inter-component $2$-gon is $\chi$-maximizing]\label{lem:Inter2gon}
Let $L$ be a $2$-component alternating link.  If an alternating link diagram $D_L$ has an inter-component $2$-gon (with changing the orientation of a link component if necessary), then a splice $\Sj$ applied to the inter-component $2$-gon is one of the splices that obtains a state producing a state surface with the maximal Euler characteristic $\chi(L)$ among spanning surfaces of $L$.  

Further,  for $D_L$, applying $\Sj$ first still obtains a state surface with the maximal $\chi(L)$ among the spanning surfaces of $L$.   
\end{lemma}
\begin{proof}
Since the inter-component $2$-gon consists of two crossings between different components, it cannot be a $1$-gon crossing, i.e., a crossing between two paths on the same component.  Hence, if the Adams-Kindred algorithm (Definition~\ref{def:AKalg}) is applied to $D_L$ with the minimal $m$-gons, the inter-component $2$-gon still exists after completing the step for $m=1$, i.e., the step for $1$-gons.    Therefore, at the step for $m=2$, a crossing corresponding to the claimed $\Sj$ appears, which contributes to a state obtaining a spanning state surface with the maximal Euler characteristic.    

Reviewing this splicing process, the sequence of splices is written as follows: 
\begin{enumerate}
\item Apply $m=1$ steps, which do not involve the inter-component $2$-gon, \label{lem:step:m=1}
\item Apply $\Sj$ on a crossing of the two crossings of the inter-component $2$-gon, \label{lem:step:sj}
\item Apply the remaining splices starting from $m=1$ steps.   \label{lem:step:remain}
\end{enumerate}
If we instead apply (\ref{lem:step:sj}) $\Sj$ at the beginning, followed by (\ref{lem:step:m=1}) and (\ref{lem:step:remain}) in order, we have the same state surface of the maximal Euler characteristic.  
 
By Fact~\ref{FactAKMax}, such a state surface realizes the maximal Euler characteristic
among all spanning surfaces of $L$, 
which completes the proof.  
\end{proof}
\begin{fact}[Adams--Kindred {\cite{AdamsKindred2013}}]\label{FactAKMax}
For any alternating link, a state surface with the maximal Euler characteristic
realizes the maximal Euler characteristic among all (state or nonstate) spanning surfaces.
\end{fact}

\begin{notation}
We reinterpret steps as a process to construct a state $S_u$ in the context of computing  $u^-(L)$, by reading each $\ri^-$ as a $\Ts$, which contributes to $\# \Ts^{\ri^-}$ or $\# {\Ts^{\ri^-}}'$.    
We define the following: 
\begin{enumerate}
\item $\# \Ts$ : the number of splices of type $T_{\text{split}}$ in a sequence realizing $u^- (L)$, 
\item \(\# {\Ts^{\ri^-}}'\) : the number of splices of type $\ri^-$ from $D_L$ to $D_K$, and 
\item $\# \Ts^{\ri^-}(D_K)$ : the number of splices of type $\ri^-$ in a sequence realizing $u^-(D_K)$.  
\end{enumerate}
\end{notation}

\paragraph{\bf Counting identities.}
Let $n$ be the number of a given diagram $D_L$ of an alternating link $L$.  

\noindent For $u^-$:
\begin{enumerate}
\item $|S_u|$ counts state circles as  one  plus $\# T_{\text{split}}$;   
\item $\# T_{\text{split}}$ decomposes into contributions from $D_K$ and the intermediate step (\ref{step:ri}) from $D_L$; 
\item $n$ equals the total splices contributing to $\# T_{\text{split}}+u^-(D_K)$ with a single $\Sj$; 
\item $u^-(D_K)$ is the number of $S^-$, which equals to $u^-(D_L)$ (Definition~\ref{def:uL}, \cite{Itotakimura2018, ItoTakimura2020v}).    
\end{enumerate}
\begin{align}
|S_u| &= \# \Ts^{\ri^-} + 1.  \label{eq:SuCount}\\
\#\Ts &= \#\Ts^{\ri^-}(D_K) + \#{\Ts^{\ri^-}}' .  
\label{eq:SplitSumU}\\
n &= \#{\Ts^{\ri^-}}' + u^-(D_K) + 1. 
\label{eq:nDecompU}\\
u^-(D_L)&=u^-(D_K).
\label{eq:uDLDK}
\end{align}
This (\ref{eq:uDLDK}) follows from Lemma~\ref{lem:UnifiedCounting}(2).    

For $u^-_2$: 
Similar reasons apply, except that $u^-_2 (D_L)$ differs by one because it includes the initial $\Sj$, which is an application to an inter-component $2$-gon.    
\begin{align}
\# \Ts &= \# \Ts^{\ri^-}(D_K) + \# {\Ts^{\ri^-}}' .
\label{eq:SplitSumU2}\\
n&= \# {\Ts^{\ri^-}}' +u^-_2(D_L).
\label{eq:nDecompU2}\\
u^-_2(D_L) &= u^-(D_K)+1 \qquad (Lemma~\ref{lem:UnifiedCounting}~(2)).  
\label{eq:u2DLDK}
\end{align}
Connected-sum case: 
Counting circles is adjusted by $\ell -1$ due to overlaps, and negative factors correspond to the nonnegative integer $[m]$.  
\begin{align}
|S_{D_K}|&=|S_{D_{K_1}}| + |S_{D_{K_2}}| + \cdots + |S_{D_{K_{\ell}}}|-(\ell-1).   \label{eq:SDKSum} \\
\# \Ts^{\ri^-}&=\#\Ts^{\ri^-}(D_{K_1})+\#\Ts^{\ri^-}(D_{K_2}) + \cdots +\#\Ts^{\ri^-}(D_{K_l}).
  \label{eq:TsplitSum} \\
|S_{D_K}|&=|S_u|+[m], 
\label{eq:SvsSuM}
\end{align}
where $[m]$ accounts for the contribution of negative factors as defined in Definition~\ref{def:mdivisor}.  
\begin{fact}[\cite{ItoTakimura2020v}]\label{RealizingUK}
For any knot $K$, 
there exists a splice sequence of $S^-$'s and $\ri^-$'s realizing $u^-(K)$.  
\end{fact} 
\paragraph{Clark's inequality}
\begin{fact}[{Clark \cite{Clark1978}}]\label{FactClark}
For any knot $K$,  
\[
C(K) \le 2g(K) + 1.  
\]
\end{fact}
We now record a Clark-type inequality for links under the convention adopted in Definition~\ref{def:genus}.    
\begin{proposition}\label{prop:ExtendClark}
For any $n$-component link $L$, let $g(L)$ denote the orientable genus of $L$
and let $C(L)$ be the non-orientable genus defined
via the Euler characteristic of connected spanning surfaces
as in Definition~\ref{def:genus}.  
Then 
\begin{equation}\label{eq:ExtendClark}
C(L) \le 2g(L) + 1.  
\end{equation}
\end{proposition}
\begin{remark}\label{rmk:ExtendClark}
Proposition~\ref{prop:ExtendClark} justifies Definition~\ref{def:genus}.    
This is because for any $n$-component link,  
the maximal Euler characteristic $\chi(L)$ satisfies
$\chi(L)=2-C(L)-n$ when $C(L)\le 2g(L)$ (equivalently, $C(L) \neq 2g(L)+1$), and
$\chi(L)=2-2g(L)-n$ when $C(L)>2g(L)$ (equivalently,   $C(L)=2g(L)+1$) under the convention of Definition~\ref{def:genus}.  
Thus we use the quantity $\min\{C(L),2g(L)\}$, which naturally encodes
both cases in a unified way.
\end{remark}

\noindent{\it Proof of Proposition~\ref{prop:ExtendClark}.}
Let $L$ be an $n$-component link.  
If $\Sigma$ is an orientable surface of minimal genus $g(L)$ spanning $L$, then $\chi(\Sigma)=2-2g(L)-n$.  Consider  changing the surface to a non-orientable surface by applying the first Reidemeister move on a link diagram of $L$ (see Figure~\ref{fig:splice2}~(d) as a  projection).  By  adding a trivial loop, we add a non-orientable handle to the surface $\Sigma$, and let $\Sigma'$ denote the resulting surface whose $\chi(\Sigma)$ satisfies  $\chi(\Sigma')=\chi(\Sigma)-1=1-2g(L)-n$.  Since $C(L)$ is the minimum number of $2-n-\chi$ where $\chi$ is the maximal Euler characteristic among those of non-orientable surfaces, each of which is bounded by $L$,      
\begin{align*}
C(L) &\le 2-n-\chi(\Sigma') \\
&=2-n-(1-2g(L)-n)\\
&=1+2g(L). 
\end{align*}
\hfill$\Box$
\begin{remark}\label{rmk:NotationCap}
Let $L$ be an $n$-component link.  If $\beta_1 (L)$ denotes the minimal first Betti number among non-orientable surfaces bounded by $L$, then a similar but slightly different inequality holds for $\beta_1 (L)$.  The Euler characteristic $\chi(L)$ realizing $\beta_1 (L)$ satisfies $\chi(L)=1-\beta_1 (L)$.  Hence, taking the same $\Sigma$ and $\Sigma'$ as in Proposition~\ref{prop:ExtendClark}, we have   
\begin{align*}
\beta_1(L) &\le 1-\chi(\Sigma') \\
&= 1 - (\chi(\Sigma)-1)\\
&=2 - \chi(\Sigma)\\
&=2 - (2-2g(L)-n)\\
&= 2g(L) + n,  
\end{align*}
which implies
\begin{equation}\label{eq:BettiClark}
\beta_1(L) \le 2g(L) + n.  
\end{equation}
Since $\chi(L)$ is the Euler characteristic realizing $\beta_1 (L)$, we have $C(L)=2-\chi(L)-n$ and $\beta_1 (L)=1-\chi(L)$, and thus the difference between the two inequalities (\ref{eq:ExtendClark}) and (\ref{eq:BettiClark}) is $n-1$. 
Note that some references use $\beta_1 (L)$ as the definition of $C(L)$, so caution is required.  
\end{remark}
\subsection{Case $u^-$ (prime)}\label{sec:case-u} 
First, we treat the case \emph{$D_K$ is prime} where \(D_K\) is the knot diagram obtained from \(D_L\) by the initial \(\Sj\) (Lemma~\ref{lem:UniqueSjoin}), applied to an inter-component $2$-gon, and the subsequent \(\ri^-\)-reductions.  
Let $n$ be the number of crossings of an alternating diagram $D_L$ of a link $L$.  As (\ref{step:J}) in Framework~\ref{framework}, apply $\Sj$ first at an inter-component $2$-gon  (Lemmas~\ref{lem:UniqueSjoin} and \ref{lem:Inter2gon}), and let $D_K$ be the resulting knot diagram.  Remove all $1$-gons by $\ri^-$'s $\# {\Ts^{\ri^-}}'$ times to obtain prime or a connected sum of prime alternating diagram (\ref{step:ri} of Framework~\ref{framework}).    
By Fact~\ref{thm:ItoTakimuraUC}, we find a splice sequence realizing 
$u^- (D_K)=u^-(K)$ (\ref{step:U} of Framework~\ref{framework}), which is interpreted (by replacing $\ri^-$ with $T_{\text{split}}$) as: 
\begin{equation}\label{eq:KvsU}
|S_{D_K}| = |S_u|,  
\end{equation}
which realizing the maximal Euler characteristic for $K$ 
(for the detail of the reason, see \cite{ItoTakimura2020v}).  

Suppose that, for $D_L$, we have a state $S_L$ implying a corresponding spanning surface  with the maximal Euler characteristic $\chi(L) = |S_L| - n$ (Definition~\ref{def:AKalg}).   
Then, 
\begin{align}
    \chi(L) &= |S_L| - n 
     \nonumber\\
    &= |S_{D_K}| + \# {\Ts^{\ri^-}}' - n \quad(\because {\textrm{the definition of }}D_K) \nonumber\\
    &= |S_u| + \# {\Ts^{\ri^-}}' - n \quad(\because \textup{(\ref{eq:KvsU})}) \nonumber\\
    &= 1 + \# \Ts^{\ri^-}(D_K) + \# {\Ts^{\ri^-}}' - n \quad(\because (\ref{eq:SuCount})) \nonumber\\
    &= 1 + \# \Ts - n \quad(\because  \textup{(\ref{eq:SplitSumU})}) \nonumber\\
    &= 1 + \# \Ts - (\# \Ts + u^-(D_K) + 1 )  \quad(\because  \textup{(\ref{eq:nDecompU})}) \nonumber\\
    &= 1 + \# T_{\text{split}} - (\# T_{\text{split}} + u^-(D_L) + 1 ) \quad(\because (\ref{eq:uDLDK})) \nonumber\\
    &= - u^-(D_L).  \label{eq:MaxEqU}
\end{align} 
If $\Sigma_{AK}$ is non-orientable, then since $L$ has two components,
Definition~\ref{def:genus} gives
\[
\chi(\Sigma_{AK}) = 2 - C(L) - 2 = -C(L).
\]
Together with $\chi(L)=\chi(\Sigma_{AK})=-u^-(D_L)$ obtained above,
this implies $C(L)$ $=$ $u^-(D_L)$.  

If \(\Sigma_{AK}\) is orientable, 
Proposition~\ref{prop:ExtendClark} implies \[2g(L)<C(L) \Leftrightarrow C(L)=2g(L)+1, \]
and then $g(L)=\frac{u^-(L)}{2}$ and $C(L)=u^-(L)+1$ which is realized by replacing \(T_{\text{split}}\) at a crossing with the splice \(S^- \) (e.g.,  Figure~\ref{fig:orientable}).    
Here we will find a surface with the maximal Euler characteristic among non-orientable spanning surfaces of $L$.  The orientable surface \(\Sigma_{AK}\) we have already been found is changed into non-orientable surface by replacing a splice \(T_{\text{split}}\) with \(\Sj\) at the same crossing as Figure~\ref{fig:orientable}.   This single surgery changes the Euler characteristic \(\chi(L)\) by $-1$  
By 
\(\chi(L)-1=-2g(L)-1 = -C(L), \) and \(\chi(L)= -u^-(L)\) (\ref{eq:MaxEqU}), 
\[C(L)-1=u^-(L),  \]
\[ g(L)=\frac{u^-(L)}{2}.\]
\begin{figure}[htpb]
\includegraphics[width=8cm]{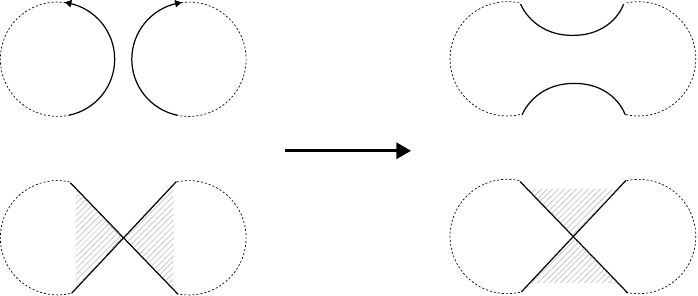}
\caption{Changing a splice type}
\label{fig:orientable}
\end{figure}

\subsection{Case: $u^-$ (non-prime)}
Next, we suppose that \(D_K\) is non-prime, hence decomposes as a connected sum of prime alternating knot diagrams in the sense of Menasco \cite{Menasco1984}.  
If \(D_K\) is non-prime, then $K$ is a connected sum of prime knots.  Moreover, a result of Menasco \cite{Menasco1984}, $D_K$ is the (obvious) connected sum of some prime alternating knot diagrams, each called a \emph{factor}.   If \(D_K\) contains no $1$-gon, then, without loss of generality, we may assume that every factor also contains no $1$-gon (if a factor has a $1$-gon with respect to a crossing $c$, then this $c$ is a nugatory crossing.).  
Let \(\ell\) be the number of factors.   We traditionally represent \(D_K=D_{K_1}\#D_{K_2}\#\cdots\#D_{K_l}\) using the symbol $\#$ to  indicate a connected sum operation.  
For each divisor \(D_{K_i}(1\leq i\leq \ell)\), applying Fact~\ref{thm:ItoTakimuraUC}, we obtain \(\Sigma_u(D_{K_i}),u^-(D_{K_i})\).  
Then \(\Sigma_u(D_K),u^-(D_K)\) is computed as follows: 
\[
\Sigma_u(D_K) = \Sigma_u(D_{K_1}) \# \Sigma_u(D_{K_2}) \# \cdots \# \Sigma_u(D_{K_{\ell}}), 
\]
and, noting the additivity of $u^-$ \cite[Proposition~3]{Itotakimura2018}, 
\[
u^-(D_K) = u^-(D_{K_1}) + u^-(D_{K_2}) + \cdots + u^-(D_{K_{\ell}}).
\]
Note that for each term $u^-(D_{K_i})$, the number of state circle is given by \[|S_{u^-(D_{K_i})}|
=\# \Ts^{\ri^-} +1\]
(see (\ref{eq:SuCount})).  

The number of circles in $S_{D_K}$ is computed using state circles $|S_{D_{K_i}}|$ ($1 \le i \le \ell$) as (\ref{eq:SDKSum}) (note that this formula uses both the number $|S_{D_{K_i}}|$ and the number of factors $\ell$): 
\begin{equation*}
|S_{D_K}|=|S_{D_{K_1}}| + |S_{D_{K_2}}| + \cdots + |S_{D_{K_{\ell}}}|-(\ell-1). 
\end{equation*}
Here, each term $|S_{D_{K_i}}|$ is  the number of state circles in the corresponding to the factor diagram, and the subtraction of $(\ell -1)$ accounts for the overlaps removed when forming the connected sum.  This adjustment ensures that the total count reflects the actual number of circles in the combined diagram.  

Since the number of splices of type $\Ts$ corresponding to $\ri^-$ in a splice sequence realizing  $u^-(D_K)$ equals the sum of those in each splice sequence realizing $u^-(D_{K_i})$ ($1 \le i \le \ell$),  we have (\ref{eq:TsplitSum}):
\begin{equation*}
    \# \Ts^{\ri^-} =\#\Ts^{\ri^-}(D_{K_1})+\#\Ts^{\ri^-}(D_{K_2}) + \cdots +\#\Ts^{\ri^-}(D_{K_l}).  
\end{equation*}
Further, the number $|S_{D_K}|$ of circles in $S_{D_K}$ satisfies (\ref{eq:SvsSuM}):   
\begin{equation*}
    |S_{D_K}|=|S_u|+[m], 
\end{equation*}
where $[m]$ accounts for the contribution of negative factors as defined in Definition~\ref{def:mdivisor}.  

Here, let $n$ be the number of crossings in $D_L$, and let $|S_L|$ be the number of circles in $S_L$ obtained by applying the Adams-Kindred algorithm (Definition~\ref{def:AKalg}) to splice the $n$ crossings.  It is elementary to see that the maximal Euler characteristic $\chi(L)$ of the link $L$ is given by  
$
\chi(L) = |S_L|-n$.  

Hence, we compute: 
\begin{align*}
    \chi(L) &= |S_L|-n 
    \\
    &= |S_{D_K}| + \#{\Ts^{\ri^-}}' - n \\
&\qquad\qquad (\because {\textrm{the definition of $D_K$ from $D_L$ applying $\ri^-$ exactly $\# T_{\textrm{split}}'$ times}})\\
    &= |S_{D_{K_1}}| + |S_{D_{K_2}}| + \cdots + |S_{D_{K_l}}| - (l-1) + \#{\Ts^{\ri^-}}' - n \quad(\because (\ref{eq:SDKSum}))\\
    &=|S_{u^-(D_{K_1})}| + |S_{u^-(D_{K_2})}| + \cdots + |S_{u^-(D_{K_l})}| \\
    &\quad - (l-1) + \#{\Ts^{\ri^-}}' - n + [m]\quad(\because(\ref{eq:SvsSuM}))\\
    &= (1+\#\Ts^{\ri^-}(D_{K_1})) + (1+\#\Ts^{\ri^-}(D_{K_2})) + \cdots + (1+\#\Ts^{\ri^-}(D_{K_l})) \\
    &\quad- (l-1) + \#{\Ts^{\ri^-}}' - n + [m] \quad(\because(\ref{eq:SuCount}))\\
    &= 1 + \#\Ts^{\ri^-}(D_K) + \#{\Ts^{\ri^-}}' - n + [m] \quad(\because(\ref{eq:TsplitSum}))\\
    &= 1 + \# \Ts - n +[m]\quad (\because(\ref{eq:SplitSumU}))\\
    &= 1 + \# \Ts - (\# \Ts + u^-(D_L) + 1)+[m] \quad(\because(\ref{eq:nDecompU}))\\
    &= - u^-(D_L)+[m].
\end{align*}
   For the number $m$ of negative factors $(0 \leq m \leq l)$, we have 
\begin{align*}
    C(L)&=u^-(D_{K_1}) + u^-(D_{K_2})+ \cdots +u^-(D_{K_l})-[m]\\
     &= u^-(D_K)-[m]\\
     &=u^-(D_L)-[m].  
     \end{align*}
     
If \(\Sigma_{AK}\) is non-orientable, 
     then $\chi(L)=-C(L)=-u^-(L)+[m]$.  
  
If \(\Sigma_{AK}\) is orientable, 
Proposition~\ref{prop:ExtendClark} implies \[2g(L)<C(L) \Leftrightarrow C(L)=2g(L)+1, \]
and then $g(L)=\frac{u^-(L)}{2}$ and $C(L)=u^-(L)+1$ which is  realized by replacing \(T_{\text{split}}\) at a crossing with the splice \(S^- \) (e.g.,  Figure~\ref{fig:orientable}).    
 This replacement decreases the number of disks by $1$, which implies that $\chi(L)$ decreases by $1 $.  
 Since 
 \(\chi(L)-1=-2g(L)-1 = -C(L), \chi(L)= -u^-(L)+[m]\), we have 
    \[C(L)=u^-(L) +1 -[m] , \]
   \[ g(L)=\frac{u^-(L)-[m]}{2}.  \]

\subsection{Case: $u^-_2$}\label{u2Proof}
The computation for \(u^-_2\) is reduced to the \(u^-\)-case by 
\[
u^-_2(D_L)=u^-(D_L)+1 \qquad ({\text {Lemma~\ref{lem:UnifiedCounting}(2)}}).
\]
Thus, the maximal Euler characteristic $\chi(L)$ satisfies
\[
\chi(L)=1-u^-_2(D_L)
\]
in the prime case, and by the same reason, we have 
\[
\chi(L)=1-u^-_2(D_L)+[m]
\]
in the non-prime case.  Therefore, the conclusions for \(u^-_2\) follow from the corresponding \(u^-\)-cases proved above, with the same correction term \([m]\) in the connected-sum case.

The arguments for the prime and non-prime cases are the same as those in the $u^-$-case and are therefore omitted.  The details are left to the reader.  

\section{Proof of Theorem~\ref{Submainthm}}\label{sec:proof}
Let $L$ be a non-splittable two-component link.  For an alternating diagram $D_L$ of $L$, $D_L$ contains at least one crossing formed by two distinct components.
Applying the local operation shown in Figure~\ref{fig:tsuika}, there are two possible links.  Then $L'$ denotes the link whose spanning surface has the larger Euler characteristic among them.
For each of $L$ and $L'$, let $\Sigma(\cdot)$ be a spanning surface realizing the maximal Euler characteristic.  Then the following equality holds:
\begin{lemma}\label{lem:plusBigon}
\[\chi(\Sigma(L)) = \chi(\Sigma(L')).\]
\end{lemma}
\begin{proof}
Let $|S_\cdot|$ ($\cdot =L$ or $L'$) denote the number of circles of a state $S_\cdot$ of $D_{\cdot}$.  Let $n$ denote the number of crossings in $D_L$.  Then the maximal Euler characteristic of a spanning surface is given by
\[
\chi(\Sigma(\cdot)) = |S_\cdot| - n .
\]
Since $L$ is a non-splittable two-component link, we choose any crossing consisting of two distinct components of $L$, and apply the operation shown in Figure~\ref{fig:tsuika}, which adds two $2$-gons.
For the resulting link $L'$, the corresponding spanning surface $\Sigma(L')$ has exactly two more state circles than $\Sigma(L)$, while the number of crossings also increases by two.
Since adding a single crossing increases the number of state circles  by at most one, the Euler characteristic obtained by this operation is already maximal.
Therefore,
\[
\chi(\Sigma(L')) = |S_{L'}|-(n+2) = (|S_L| + 2) - (n + 2) = |S_L| - n = \chi(\Sigma(L)).
\]
\end{proof}

\noindent{\emph{Proof of Theorem~\ref{thm:uL}}.}
Given a non-splittable two-component link $L$ and an alternating diagram $D_L$, the above construction yields a link $L'$ admitting an alternating  diagram with an inter-component $2$-gon without changing the maximal Euler characteristic.  
Hence, without loss of generality, we may suppose that the given link $L$ admits an alternating diagram with an inter-component $2$-gon.    Then $L$ satisfies the assumption of Proposition~\ref{Subprop}.  

Applying Proposition~\ref{Subprop} to $L$, we obtain the conclusion of Theorem~\ref{thm:uL}.
\hfill$\Box$

\begin{figure}[htbp]
\includegraphics[width=10cm]{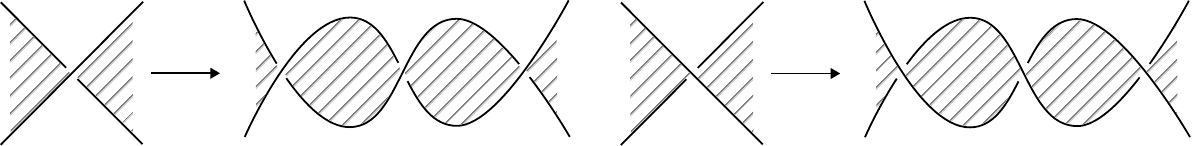}
\caption{Addition of $2$-gons}\label{fig:tsuika}
\end{figure}

\section{Table}\label{sec:table}
Table~\ref{tab:values} lists the values of the splice-unknotting numbers
$u^-(L)$ and $u^-_2(L)$ of links, the crosscap number $C(L)$, and the maximal Euler characteristic
$\chi(L)$ for several two-component prime alternating links from the link table \cite{RolfsenBook2003}.

\begin{table}[htbp]
\centering
\renewcommand{\arraystretch}{1.5}
\begin{tabular}{rrrrr}
\hline
$L$ & $u^-(L)$ & $u^-_2(L)$ & $C(L)$ & $\chi(L)$  \\ \hline
$2^2_1$ & 0 & 1 & 1 & 0  \\
$4^2_1$ & 0 & 1 & 1 & 0  \\
$5^2_1$ & 1 & 2 & 1 & $-1$  \\
$6^2_1$ & 0 & 1 & 1 & 0  \\
$6^2_2$ & 1 & 2 & 1 & $-1$  \\
$6^2_3$ & 2 & 3 & 2 & $-2$  \\
$7^2_1$ & 1 & 2 & 1 & $-1$  \\
$7^2_2$ & 2 & 3 & 2 & $-2$  \\
$7^2_3$ & 2 & 3 & 2 & $-2$  \\
$7^2_4$ & 1 & 2 & 1 & $-1$  \\
$7^2_5$ & 2 & 3 & 2 & $-2$  \\
$7^2_6$ & 2 & 3 & 2 & $-2$  \\ \hline
\end{tabular}
\caption{Splice-unknotting numbers, non-orientable genus, and related values for prime two-component alternating links.  
}
\label{tab:values}
\end{table}

\section{Examples}\label{sec:example}
We illustrate the splice operations and the computation of $u^-(L)$
for a few two-component alternating links.

\begin{figure}[htbp]
    \centering
    \includegraphics[width=\linewidth]{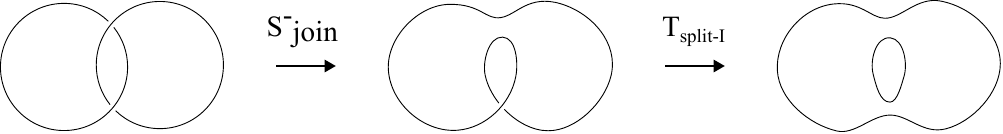}
    \caption{A splice sequence for the link $2^2_1$.  Here we use the notation $T_{\text{split-I}}$ only for typographical convenience, since writing ``$\Ts$ corresponding to $\ri^-$'' above the arrow is impractical in this figure;   no additional operation is intended.  }
    \label{fig:ex-2-1}
\end{figure}

\begin{figure}[htbp]
    \centering
    \includegraphics[width=\linewidth]{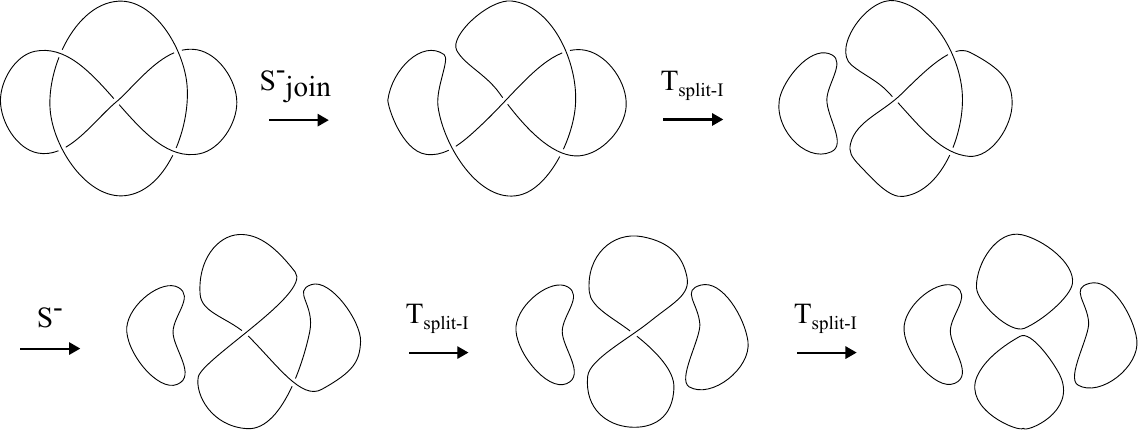}
    \caption{A splice sequence for the link $5^2_1$.  Here we use the notation $T_{\text{split-I}}$ only for typographical convenience, since writing ``$\Ts$ corresponding to $\ri^-$'' above the arrow is impractical in this figure;  no additional operation is intended.  }
    \label{fig:ex-5-1}
\end{figure}

An alternating diagram of the two-component link $7^2_6$ admits no inter-component $2$-gon.
Hence we first add $2$-gons as in Figure~\ref{fig:tsuika} before applying splice operations.
Figure~\ref{fig:7_6max} shows a choice of added $2$-gons that realizes the maximal
Euler characteristic, while Figure~\ref{fig:7_6} shows a modification that does not.

\begin{figure}[htpb]
    \centering
    \includegraphics[width=\linewidth]{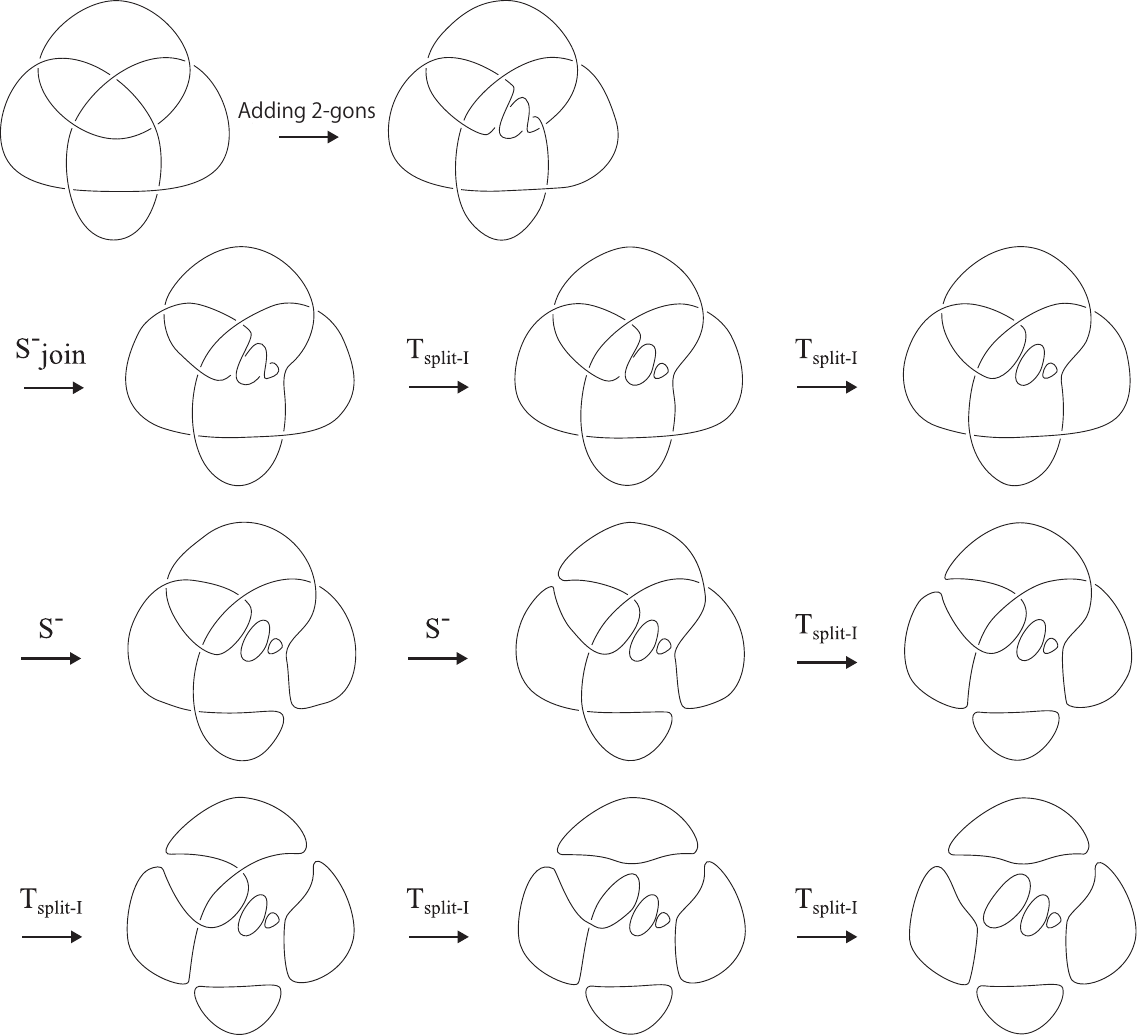}
    \caption{A splice sequence of a modified diagram obtained from a diagram of $7^2_6$ realizing the maximal Euler characteristic.  Here we use the notation $T_{\text{split-I}}$ only for typographical convenience, since writing ``$\Ts$ corresponding to $\ri^-$'' above the arrow is impractical in this figure;   no additional operation is intended.  }
    \label{fig:7_6max}
\end{figure}

\begin{figure}[htpb]
    \centering
    \includegraphics[width=8.5cm]{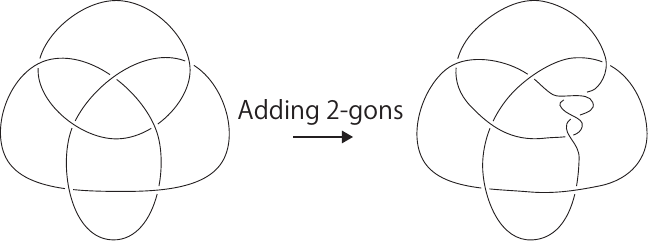}
    \caption{A modification of $7^2_6$ that does not realize the maximal Euler characteristic}
    \label{fig:7_6}
\end{figure}

\bibliographystyle{amsplain}
\bibliography{TwoCRef}

@article {Seifert1935,
    AUTHOR = {Seifert, H.},
     TITLE = {\"{U}ber das {G}eschlecht von {K}noten},
   JOURNAL = {Math. Ann.},
  FJOURNAL = {Mathematische Annalen},
    VOLUME = {110},
      YEAR = {1935},
    NUMBER = {1},
     PAGES = {571--592},
      ISSN = {0025-5831,1432-1807},
   MRCLASS = {DML},
  MRNUMBER = {1512955},
       DOI = {10.1007/BF01448044},
       URL = {https://doi.org/10.1007/BF01448044},
}

@article {Ozawa2011,
    AUTHOR = {Ozawa, Makoto},
     TITLE = {Essential state surfaces for knots and links},
   JOURNAL = {J. Aust. Math. Soc.},
  FJOURNAL = {Journal of the Australian Mathematical Society},
    VOLUME = {91},
      YEAR = {2011},
    NUMBER = {3},
     PAGES = {391--404},
      ISSN = {1446-7887,1446-8107},
   MRCLASS = {57M25 (57M20 57M27)},
  MRNUMBER = {2900614},
MRREVIEWER = {Sergej\ V.\ Matveev},
       DOI = {10.1017/S1446788712000055},
       URL = {https://doi.org/10.1017/S1446788712000055},
}

@article {PabiniakPrzytyckiSazdanovic2009,
    AUTHOR = {Pabiniak, Milena D. and Przytycki, J\'{o}zef H. and
              Sazdanovi\'{c}, Radmila},
     TITLE = {On the first group of the chromatic cohomology of graphs},
   JOURNAL = {Geom. Dedicata},
  FJOURNAL = {Geometriae Dedicata},
    VOLUME = {140},
      YEAR = {2009},
     PAGES = {19--48},
      ISSN = {0046-5755,1572-9168},
   MRCLASS = {05C15 (57M15 57M25 57M27)},
  MRNUMBER = {2504733},
MRREVIEWER = {Sergei\ K.\ Lando},
       DOI = {10.1007/s10711-008-9307-4},
       URL = {https://doi.org/10.1007/s10711-008-9307-4},
}

@article {KalfagianniLee2016,
    AUTHOR = {Kalfagianni, Efstratia and Lee, Christine Ruey Shan},
     TITLE = {Crosscap numbers and the {J}ones polynomial},
   JOURNAL = {Adv. Math.},
  FJOURNAL = {Advances in Mathematics},
    VOLUME = {286},
      YEAR = {2016},
     PAGES = {308--337},
      ISSN = {0001-8708,1090-2082},
   MRCLASS = {57M25 (57M27)},
  MRNUMBER = {3415687},
MRREVIEWER = {Masakazu\ Teragaito},
       DOI = {10.1016/j.aim.2015.09.017},
       URL = {https://doi.org/10.1016/j.aim.2015.09.017},
}

@article {MurakamiYasuhara1995,
    AUTHOR = {Murakami, Hitoshi and Yasuhara, Akira},
     TITLE = {Crosscap number of a knot},
   JOURNAL = {Pacific J. Math.},
  FJOURNAL = {Pacific Journal of Mathematics},
    VOLUME = {171},
      YEAR = {1995},
    NUMBER = {1},
     PAGES = {261--273},
      ISSN = {0030-8730,1945-5844},
   MRCLASS = {57M25},
  MRNUMBER = {1362987},
MRREVIEWER = {Chichen\ M.\ Tsau},
       URL = {http://projecteuclid.org/euclid.pjm/1102370328},
}

@article {Murasugi1958,
    AUTHOR = {Murasugi, Kunio},
     TITLE = {On the genus of the alternating knot. {I}, {II}},
   JOURNAL = {J. Math. Soc. Japan},
  FJOURNAL = {Journal of the Mathematical Society of Japan},
    VOLUME = {10},
      YEAR = {1958},
     PAGES = {94--105, 235--248},
      ISSN = {0025-5645,1881-1167},
   MRCLASS = {55.00},
  MRNUMBER = {99664},
       DOI = {10.2969/jmsj/01010094},
       URL = {https://doi.org/10.2969/jmsj/01010094},
}

@article {Crowell1959,
    AUTHOR = {Crowell, Richard},
     TITLE = {Genus of alternating link types},
   JOURNAL = {Ann. of Math. (2)},
  FJOURNAL = {Annals of Mathematics. Second Series},
    VOLUME = {69},
      YEAR = {1959},
     PAGES = {258--275},
      ISSN = {0003-486X},
   MRCLASS = {55.00},
  MRNUMBER = {99665},
MRREVIEWER = {R.\ H.\ Fox},
       DOI = {10.2307/1970181},
       URL = {https://doi.org/10.2307/1970181},
}

@book{RolfsenBook2003,
  author    = {Dale Rolfsen},
  title     = {Knots and Links},
  series    = {Mathematics Lecture Series},
  volume    = {7},
  publisher = {AMS Chelsea Publishing},
  address   = {Providence, RI},
  year      = {2003},
  note      = {Reprint of the 1976 edition},
  mrnumber  = {0515288},
}

@article {ItoTakimura2018,
    AUTHOR = {Ito, Noboru and Takimura, Yusuke},
     TITLE = {Crosscap number and knot projections},
   JOURNAL = {Internat. J. Math.},
  FJOURNAL = {International Journal of Mathematics},
    VOLUME = {29},
      YEAR = {2018},
    NUMBER = {12},
     PAGES = {1850084, 21},
      ISSN = {0129-167X},
   MRCLASS = {57M25},
  MRNUMBER = {3883412},
MRREVIEWER = {Masakazu Teragaito},
       DOI = {10.1142/S0129167X18500842},
       URL = {https://doi.org/10.1142/S0129167X18500842},
}

@article {IchiharaMizushima2010,
    AUTHOR = {Ichihara, Kazuhiro and Mizushima, Shigeru},
     TITLE = {Crosscap numbers of pretzel knots},
   JOURNAL = {Topology Appl.},
  FJOURNAL = {Topology and its Applications},
    VOLUME = {157},
      YEAR = {2010},
    NUMBER = {1},
     PAGES = {193--201},
      ISSN = {0166-8641},
   MRCLASS = {57M25 (57M27)},
  MRNUMBER = {2556097},
MRREVIEWER = {Neil R. Nicholson},
       DOI = {10.1016/j.topol.2009.04.031},
       URL = {https://doi.org/10.1016/j.topol.2009.04.031},
}

@article {HIrasawaTeragaito2006,
    AUTHOR = {Hirasawa, Mikami and Teragaito, Masakazu},
     TITLE = {Crosscap numbers of 2-bridge knots},
   JOURNAL = {Topology},
  FJOURNAL = {Topology. An International Journal of Mathematics},
    VOLUME = {45},
      YEAR = {2006},
    NUMBER = {3},
     PAGES = {513--530},
      ISSN = {0040-9383},
   MRCLASS = {57M25},
  MRNUMBER = {2218754},
MRREVIEWER = {Bruno Martelli},
       DOI = {10.1016/j.top.2005.11.001},
       URL = {https://doi.org/10.1016/j.top.2005.11.001},
}

@article {HatcherThurston1985,
    AUTHOR = {Hatcher, A. and Thurston, W.},
     TITLE = {Incompressible surfaces in {$2$}-bridge knot complements},
   JOURNAL = {Invent. Math.},
  FJOURNAL = {Inventiones Mathematicae},
    VOLUME = {79},
      YEAR = {1985},
    NUMBER = {2},
     PAGES = {225--246},
      ISSN = {0020-9910},
   MRCLASS = {57M25},
  MRNUMBER = {778125},
MRREVIEWER = {Charles Livingston},
       DOI = {10.1007/BF01388971},
       URL = {https://doi.org/10.1007/BF01388971},
}

@article {Teragaito2004,
    AUTHOR = {Teragaito, Masakazu},
     TITLE = {Crosscap numbers of torus knots},
   JOURNAL = {Topology Appl.},
  FJOURNAL = {Topology and its Applications},
    VOLUME = {138},
      YEAR = {2004},
    NUMBER = {1-3},
     PAGES = {219--238},
      ISSN = {0166-8641},
   MRCLASS = {57M25},
  MRNUMBER = {2035482},
MRREVIEWER = {Carlo Petronio},
       DOI = {10.1016/j.topol.2003.08.004},
       URL = {https://doi.org/10.1016/j.topol.2003.08.004},
}

@article{ItoTakimura2020b,
    AUTHOR = {Ito, Noboru and Takimura, Yusuke},
     TITLE = {A lower bound of crosscap numbers of alternating knots},
   JOURNAL = {J. Knot Theory Ramifications},
  FJOURNAL = {Journal of Knot Theory and its Ramifications},
    VOLUME = {29},
      YEAR = {2020},
    NUMBER = {1},
     PAGES = {1950092, 15},
      ISSN = {0218-2165},
   MRCLASS = {57K10},
  MRNUMBER = {4079618},
MRREVIEWER = {Fengling Li},
       DOI = {10.1142/S0218216519500925},
       URL = {https://doi.org/10.1142/S0218216519500925},
}

@article {Menasco1984,
    AUTHOR = {Menasco, W.},
     TITLE = {Closed incompressible surfaces in alternating knot and link
              complements},
   JOURNAL = {Topology},
  FJOURNAL = {Topology. An International Journal of Mathematics},
    VOLUME = {23},
      YEAR = {1984},
    NUMBER = {1},
     PAGES = {37--44},
      ISSN = {0040-9383},
   MRCLASS = {57M25},
  MRNUMBER = {721450},
MRREVIEWER = {Cameron McA. Gordon},
       DOI = {10.1016/0040-9383(84)90023-5},
       URL = {https://doi.org/10.1016/0040-9383(84)90023-5},
}

@article {AdamsKindred2013,
    AUTHOR = {Adams, Colin and Kindred, Thomas},
     TITLE = {A classification of spanning surfaces for alternating links},
   JOURNAL = {Algebr. Geom. Topol.},
  FJOURNAL = {Algebraic \& Geometric Topology},
    VOLUME = {13},
      YEAR = {2013},
    NUMBER = {5},
     PAGES = {2967--3007},
      ISSN = {1472-2747},
   MRCLASS = {57M25},
  MRNUMBER = {3116310},
MRREVIEWER = {Ilya S. Kofman},
       DOI = {10.2140/agt.2013.13.2967},
       URL = {https://doi.org/10.2140/agt.2013.13.2967},
}

@article {Zhang2008,
    AUTHOR = {Zhang, Gengyu},
     TITLE = {Crosscap numbers of two-component links},
   JOURNAL = {Kyungpook Math. J.},
  FJOURNAL = {Kyungpook Mathematical Journal},
    VOLUME = {48},
      YEAR = {2008},
    NUMBER = {2},
     PAGES = {241--254},
      ISSN = {1225-6951},
   MRCLASS = {57M25 (57M27)},
  MRNUMBER = {2429312},
MRREVIEWER = {Jeff Johannes},
       DOI = {10.5666/KMJ.2008.48.2.241},
       URL = {https://doi.org/10.5666/KMJ.2008.48.2.241},
}

@incollection {Kauffman1983,
    AUTHOR = {Kauffman, Louis H.},
     TITLE = {Combinatorics and knot theory},
 BOOKTITLE = {Low-dimensional topology ({S}an {F}rancisco, {C}alif., 1981)},
    SERIES = {Contemp. Math.},
    VOLUME = {20},
     PAGES = {181--200},
 PUBLISHER = {Amer. Math. Soc., Providence, RI},
      YEAR = {1983},
   MRCLASS = {57M25 (05C10)},
  MRNUMBER = {718142},
       DOI = {10.1090/conm/020/718142},
       URL = {https://doi.org/10.1090/conm/020/718142},
}

@article {Clark1978,
    AUTHOR = {Clark, Bradd Evans},
     TITLE = {Crosscaps and knots},
   JOURNAL = {Internat. J. Math. Math. Sci.},
  FJOURNAL = {International Journal of Mathematics and Mathematical
              Sciences},
    VOLUME = {1},
      YEAR = {1978},
    NUMBER = {1},
     PAGES = {113--123},
      ISSN = {0161-1712},
   MRCLASS = {55A25},
  MRNUMBER = {478131},
MRREVIEWER = {R. J. Daverman},
       DOI = {10.1155/S0161171278000149},
       URL = {https://doi.org/10.1155/S0161171278000149},
}

@article {ItoTakimura2020v,
    AUTHOR = {Ito, Noboru and Takimura, Yusuke},
     TITLE = {Crosscap number of knots and volume bounds},
   JOURNAL = {Internat. J. Math.},
  FJOURNAL = {International Journal of Mathematics},
    VOLUME = {31},
      YEAR = {2020},
    NUMBER = {13},
     PAGES = {2050111, 33},
      ISSN = {0129-167X},
   MRCLASS = {57K10 (57K32)},
  MRNUMBER = {4192453},
MRREVIEWER = {Fengling Li},
       DOI = {10.1142/S0129167X20501116},
       URL = {https://doi.org/10.1142/S0129167X20501116},
}

@article {Kindred2020,
    AUTHOR = {Kindred, Thomas},
     TITLE = {Crosscap numbers of alternating knots via unknotting splices},
   JOURNAL = {Internat. J. Math.},
  FJOURNAL = {International Journal of Mathematics},
    VOLUME = {31},
      YEAR = {2020},
    NUMBER = {7},
     PAGES = {2050057, 30},
      ISSN = {0129-167X},
   MRCLASS = {57K10},
  MRNUMBER = {4123946},
MRREVIEWER = {Fengling Li},
       DOI = {10.1142/S0129167X20500573},
       URL = {https://doi.org/10.1142/S0129167X20500573},
}
\end{document}